\documentclass[reqno]{amsart}
\usepackage{amssymb}
\usepackage{mathrsfs}
\usepackage{amsmath,amstext,amsfonts,verbatim}
\usepackage{color}
\usepackage[all]{xy}
\usepackage{graphics}
\usepackage{graphicx}
\usepackage{exscale}
\usepackage{relsize}
\usepackage[latin1]{inputenc}
\usepackage{amsmath,amsthm,amssymb}

\newtheorem{thm}{Theorem}
\numberwithin{thm}{section}
\newtheorem{lemma}[thm]{Lemma}
\newtheorem{cor}[thm]{Corollary}

\newtheorem{rem}[thm]{Remark}
\newtheorem{example}[thm]{Example}

\newcommand{\neweq}[1]{\begin{equation}\label{#1}}

\def\phi{\varphi}

\def\incep{\left\{\begin{array}{cl} }
 \def\termin{\end{array}\right. }
\def\2af{2^*_\alpha}

\title[Spectral picture of Toeplitz operators]{\textbf{The spectral picture of Bergman-Toeplitz operators with harmonic polynomial symbols}}

\author{Kunyu Guo}
\address{School of Mathematical Sciences, Fudan University, Shanghai, 200433, PR \ China.}
\email{kyguo@fudan.edu.cn}

\author{Xianfeng  Zhao}
\address{College of Mathematics and Statistics, Chongqing University,  Chongqing, 401331,  PR \  China.}
\email{xianfengzhao@cqu.edu.cn}

\author{Dechao Zheng}
\address{Center of  Mathematics,  Chongqing  University,  Chongqing,   401331, PR \ China   and
Department of   Mathematics, Vanderbilt  University,   Nashville, TN  37240, United  States.}
\email{dechao.zheng@vanderbilt.edu}


\begin{document}

\maketitle
\begin{center}{}\end{center}

\begin{abstract}
This paper shows some new phenomenon in the spectral theory of Toeplitz operators on the Bergman space, which is considerably different from that of Toeplitz operators on the Hardy space. On the one hand, we prove that the spectrum of the Toeplitz operator with symbol ${\overline{z}+p}$ is always connected for every polynomial $p$ with degree less than $3$. On the other hand, we show that for each integer $k$ greater than $2$, there exists a  polynomial $p$ of degree $k$ such that the spectrum of the  Toeplitz operator with symbol ${\overline{z}+p}$ has at least one isolated point but has at most finitely many isolated points. Then these results are applied to obtain a new class of non-hyponormal Toeplitz operators with bounded harmonic symbols on the Bergman space for which Weyl's theorem holds.
\end{abstract}

\tableofcontents

\section{Introduction}

Let $dA$ denote the Lebesgue  measure on the open unit disk $\mathbb D$ in the complex plane $\mathbb C$, normalized so that the measure of the disk
$\mathbb D$ is $1$. The complex space $L^2(\mathbb D, dA)$ is a Hilbert space with the inner product
$$\langle f, g\rangle=\int_{\mathbb D} f(z)\overline{g(z)}dA(z).$$
The Bergman space $L_a^2$ is the set of those functions in $L^2(\mathbb D, dA)$ that are analytic on $\mathbb D$. Thus the Bergman space is a closed subspace of $L^2(\mathbb D, dA)$ and so there is an orthogonal projection $P$ from $L^2(\mathbb D, dA)$ onto $L_a^2$.
For $\varphi\in L^\infty(\mathbb D, dA)$, the  Toeplitz operator $T_{\varphi}$ with symbol $\varphi$ on the Bergman space (or ``Bergman-Toeplitz operator") is  defined by $$T_\varphi f=P(\varphi f)$$
for $f$ in the Bergman space $L_a^2$.

In general, the behaviour of these operators may be quite different from that of the Toeplitz operators on the Hardy space. For example, there are many nontrivial compact Toeplitz operators on the Bergman space \cite{Ax, AxZ, Str1, Zhu1}. But there is no a nontrivial compact Toeplitz operator on the Hardy space \cite{Dou}.   On the other hand, there is no a nontrivial compact Toeplitz operator with a bounded harmonic symbol on the Bergman space \cite{AxZ} and the harmonic extension gives a natural corresponding relationship between the bounded functions on the unit circle and the bounded harmonic functions on the unit disk \cite{Dur}. Thus  Toeplitz operators on the Bergman space with harmonic symbols behave quite similarly to those on the Hardy space \cite{McS}. As a fundamental problem concerning Toeplitz operators is to determine the spectra in terms of the properties of their symbols, it is natural to study the spectra of  Toeplitz operators with bounded harmonic symbols on the Bergman space.

An important result about Toeplitz operators on the Hardy space is the Widom theorem  \cite{Widom1, Widom2}, which states that the spectrum of a bounded Toeplitz operator is always connected. Moreover, by means of some techniques in complex analysis and differential equation, Douglas showed that the essential spectrum of a bounded Toeplitz operator on the Hardy space is also connected \cite[Theorem 7.45]{Dou}.  On the one hand, as there are many nontrivial compact Toeplitz operators on the Bergman space,  one can easily construct a Toeplitz operator whose spectrum has isolated points and hence is disconnected. On the other hand, McDonald and Sundberg showed in  \cite{McS} that the spectrum and essential spectrum of $T_\varphi$ on the Bergman space are both  connected for $\varphi$ a bounded and real-valued harmonic function on $\mathbb D$. Moreover, they also showed that the essential spectrum of $T_\varphi$ is connected if $\varphi$ is harmonic on $\mathbb D$ and  piecewise
continuous on the boundary of the disk $\mathbb D$. These suggested the conjecture that a Toeplitz operator on the Bergman space with harmonic symbol has a connected spectrum \cite[Page 320]{Nik}. In 1979,  McDonald and Sundberg asked the question in \cite{McS} whether the essential spectrum of a  Toeplitz operator with bounded harmonic symbol is connected. About 30 years later, Sundberg and the third author gave a negative  answer to the above question and disproved the above conjecture  in \cite{SZ}. Indeed, they first constructed a rational function $q$ on $\mathbb D$ via two conformal mappings and showed that
the spectrum of the Toeplitz operator $T_{\overline{z}+q}$ has at least one isolated point and hence  is disconnected.  Based on a characterization of the essential spectrum for a certain Toeplitz algebra established in \cite{Sua} and some techniques in function algebra theory, the authors used the function $q$ quoted above to construct a bounded harmonic function $h$ such that the essential spectrum of the Toeplitz operator $T_h$  also has an isolated point \cite{SZ}. Despite considerable effort devoted to studying more about the invertibility and spectra of Bergman-Toeplitz operators with harmonic symbols, little progress has been made on this topic in recent 10 years.

 An especially important but quite difficult problem in operator theory is to determine the spectrum of a bounded linear operator. Note that the spectral structure is closely related to the invariant subspaces of  bounded linear operators. Indeed, for each bounded linear operator $T$ with disconnected spectrum, we can apply  the Riesz decomposition theorem to construct a hyperinvariant subspace of $T$, see \cite[Lemma 1.19]{Bay} or Proposition 4.11 of Chapter VII in \cite{Con}. However, there is little characterization for the topological structure of the spectrum of the Toeplitz operator with a bounded harmonic symbol, even if the symbol is the harmonic function ${\overline{z}+p}$ for  an analytic polynomial $p$. In this paper,  we will investigate the structure of the spectrum of the Toeplitz operator $T_{\overline{z}+p}$ via certain analytic properties of polynomials. The main idea is to show how  eigenvalues of  $T_{\overline{z}+p}$ depend on $p$ by solving the first order complex differential equation for eigenvectors. But nontrivial solutions of the differential equation in the Bergman space  will be subject to zeros in the unit disk of some analytic polynomials. In order to estimate the modulus of these complex zeros, our approach here is to use the theorem of the zeros of a polynomial depend continuously on its coefficients \cite{Ost,RS}.

For a function $\varphi$ bounded and analytic on the unit disk $\mathbb D$, it is well-known that the spectrum of the Toeplitz operator $T_{\varphi}$ equals the closure of the image of the unit disk under $\varphi$. In \cite{ZZ}, it was shown that if the symbol $\varphi$ is an  affine function of $z$ and $\overline{z}$, then it is also true that the spectrum of $T_{\varphi}$ equals the closure of the image of $\mathbb D$ under $\varphi$, and hence  it is a connected set. In this paper, we obtain  a characterization on the point spectra of  Toeplitz operators with certain bounded harmonic symbols, and then establish a necessary and sufficient condition for this class of Toeplitz operators to be invertible on the Bergman space, see Theorems \ref{key lemma} and \ref{invertibility} in the next section. On the one hand,  we prove in Theorem \ref{(z)+az2+bz+c} that the spectrum of the Bergman  Toeplitz operator $T_{\overline{z}+p}$ is connected for every polynomial $p$ with degree less than or equal to $2$.  On the other hand,  for each integer $k$ greater than $2$, we construct an analytic  polynomial $p$ with degree $k$ such that the spectrum of the  Toeplitz operator with symbol ${\overline{z}+p}$ has at least one isolated point but has at most finitely many isolated points, see Theorems \ref{deg>2} and \ref{deg=k} for the details. The significance of  further discussing the spectral structure of these operators is that isolated spectral points are closely related to their nontrivial invariant subspaces and  hypercyclicity \cite{Bay}.

In addition, for a bounded linear operator on a Hilbert space, the topological structure of the spectrum plays an important role in the study of its \emph{Weyl spectrum} (which will be introduced in Section 5), see \cite{Ber, Ber2, Co} and \cite{Ob} for  the classical results. It is known that the Weyl spectra of every hermitian operator and every normal operator consist precisely of all points in the spectra except the isolated eigenvalues of finite  geometric multiplicity. ``Weyl's theorem for an operator"  was first introduced by Coburn \cite{Co} in 1966, which says that the complement in the spectrum of the Weyl spectrum coincides with the isolated points of the spectrum which are eigenvalues of finite  geometric multiplicity. Moreover, Coburn showed that Weyl's theorem holds for all hyponormal operators and Hardy-Toeplitz operators \cite{Co}. Weyl type theorems  with respect to isolated points  of the spectrum of an operator were investigated  for many cases and many classes of operators.  Based on the characterizations for the spectra of Toeplitz operators in Theorems \ref{key lemma} and \ref{deg>2}, we show in Theorem \ref{weyl} that the Bergman-Toeplitz operator $T_{\overline{z}+q}$ satisfies Weyl's theorem, where $q$ is an  arbitrary function in the disk algebra $H^\infty\cap C(\overline{\mathbb D})$.

\section{Preliminary}

In this section, we first introduce some notations and include some lemmas.
 As usual, we use $\sigma(T_\varphi)$, $\sigma_{p}(T_\varphi)$ and $\sigma_e(T_\varphi)$ to denote the spectrum, point spectrum (or the set of  eigenvalues) and essential spectrum of the Toeplitz operator $T_\varphi$, respectively. Let $\mathbb N$ denote the set of  nonnegative integers.
The following lemma is useful for us to look for eigenvalues of  Toeplitz operators with some harmonic symbols \cite[Lemma 2.1]{SZ}.

\begin{lemma}\label{a formula}
For each function $f\in L_a^2$, we have
$$T_{\overline{z}}f(z)=\frac{1}{z^2}\int_0^z wf'(w)dw.$$
\end{lemma}
The next lemma is about the Fredholm theory of Toeplitz operators  with continuous symbols on the Bergman space \cite{Str, Zhu}.
\begin{lemma}\label{Fredholm index}
Suppose  that $\varphi \in C(\overline{\mathbb D})$.  Then the essential spectrum of the Toeplitz operator $T_{\varphi}$ is given by
\begin{align*}
\sigma_{e}(T_{\varphi})&=\varphi (\partial \mathbb D).
\end{align*}
Moreover, if $T_\varphi$ is a Fredholm operator, then the Fredholm index of $T_\varphi$ is given by
\begin{align*}
\mathrm{index}(T_{\varphi})&=\mathrm{dim\ ker}(T_{\varphi})-\mathrm{dim\ ker}(T^*_{\varphi})=-\mathrm{wind}\big(\varphi(\partial \mathbb D), 0\big),
\end{align*}
where $\mathrm{wind}\big(\varphi(\partial \mathbb D), 0\big)$ is the winding number of the closed oriented curve $\varphi(\partial \mathbb D)$ with respect to the origin,
which is defined  by
$$\mathrm{wind}\big(\varphi(\partial \mathbb D), 0\big)=\frac{1}{2\pi \mathrm{i}}\int_{\varphi(\partial \mathbb D)}\frac{dz}{z}.$$
\end{lemma}

We will use the following spectral picture theorem  \cite[Proposition 1.27]{Per} to analyze  isolated points in the spectrum of  a  Toeplitz operator on the Bergman space.
\begin{thm} \textbf{(Pearcy)} \label{spectral picture theorem}
Let $T$ be a bounded linear operator on a Hilbert space $\mathscr{H}$ and $H$ be  ``a hole in $\sigma_e(T)$" (which is  a bounded component of $\mathbb C\backslash \sigma_e (T)$) such that $$\mathrm{index}(T-\lambda I)=0, \ \ \ \lambda \in H,$$
then either \\
$\mathrm{(a)}$  $H\cap \sigma(T)=\varnothing$,\\
$\mathrm{(b)}$ $H\subset \sigma(T)$, or\\
$\mathrm{(c)}$ $H\cap \sigma(T)$ is a countable set of isolated eigenvalues of $T$, each having finite multiplicity.

Furthermore, the intersection of $\sigma(T)$ with the unbounded component of $\mathbb C \backslash \sigma_e(T)$ is a countable set of isolated eigenvalues of $T$, each of which has finite multiplicity.
\end{thm}

The following theorem gives a characterization for the eigenvalues of a class of Toeplitz operators with harmonic symbols on the Bergman space, which is useful for us to
study the isolated points in the spectra of Toeplitz operators with some bounded harmonic symbols.
\begin{thm}\label{key lemma}
Let $p$ be a function in $H^\infty\cap C(\overline{\mathbb D})$. Suppose that $\lambda$ is a complex number not in the essential spectrum of the Toeplitz operator $T_{\overline{z}+p}$. Then $\lambda$ is an eigenvalue of  $T_{\overline{z}+p}$ if and only if  either
$1+z[p(z)-\lambda]$ does not vanish on the unit disk or $1+z[p(z)-\lambda]$ has finitely many simple zeros $\big\{z_1, \cdots, z_k\big\}$ in the unit disk which satisfy
\begin{align}\label{eqonp}
z_j^2p'(z_j)=\frac{n_j+2}{n_j+1}
\end{align}
 for some integer $n_j\in \big\{0, 1, 2, \cdots\big\}$ with $j=1, 2, \cdots, k$.
\end{thm}
\begin{proof} Let $\lambda$ be a complex number which is not contained in the essential spectrum of the Toeplitz operator $T_{\overline{z}+p}$. We have by Lemma \ref{Fredholm index} that $1+z[p(z)-\lambda ]$ does not vanish on the unit circle $\partial \mathbb D$, which yields that $1+z[p(z)-\lambda]$ has at most finitely many zeros in the open  unit disk $\mathbb D$.

 Clearly, $\lambda$ is an eigenvalue of $T_{\overline{z}+p}$ if and only if
$$(T_{\overline{z}+p}-\lambda)f=0$$
for some nonzero function $f$ in the Bergman space $L_a^2$.
   By Lemma \ref{a formula}, we have  the following integral equation:
$$\frac{1}{z^2}\int_0^z wf'(w)dw=[\lambda-p(z)]f(z).$$
Multiplying both sides of the above equation by $z^2$ and then taking derivatives give the following first order differential equation:
\begin{align}\label{difff}
\Big(1+z[p(z)-\lambda]\Big)f'(z)=-\Big(2[p(z)-\lambda]+zp'(z)\Big)f(z).
\end{align}
Therefore, the above arguments show that $\lambda\in \sigma_p(T_{\overline{z}+p})$ if and only if $\lambda$ satisfies Equation (2) for some $f\neq 0$ in the Bergman space.

Let us first show the necessity. Suppose that
$\lambda$ is an eigenvalue of the Toeplitz operator $T_{\overline{z}+p}$. If $1+z[p(z)-\lambda ]$ has no zeros in the unit disk, then we obtain the desired conclusion. Otherwise, we may assume that the zeros of $1+z[p(z)-\lambda]$ in $\mathbb D$ are given by $\big\{z_1, z_2, \cdots, z_k\big\}$.

We first show that $z_1, z_2, \cdots, z_k$ are all simple zeros of $1+z[p(z)-\lambda ]$. If some of them, e.g., $z_j$ ($j\in \{1, 2, \cdots, k\}$) is a multiple zero, using
$$\begin{cases}
\mathlarger{1+z_j[p(z_j)-\lambda]=0},\vspace{2mm}\\
\mathlarger{p(z_j)-\lambda+z_j p'(z_j)=0},
\end{cases}$$
we get
$$z_j^2 p'(z_j)=1,$$
to obtain
\begin{align*}
2z_j[p(z_j)-\lambda]+z_j^2 p'(z_j)&= -2+1=-1.
\end{align*}
Combining this with  (\ref{difff}) gives us $f(z_j)=0$. Taking derivatives on both sides of (\ref{difff}) gives
$$(p-\lambda+zp')f'+\big[1+z(p-\lambda)\big]f''=-\big(3p'+zp''\big)f-\big[2(p-\lambda)+zp'\big]f',$$
to get $f'(z_j)=0$, since we have
 $$\begin{cases}
\mathlarger{f(z_j)=0},\vspace{1.8mm}\\
\mathlarger{1+z_j[p(z_j)-\lambda]=0},\vspace{1.8mm}\\
\mathlarger{p(z_j)-\lambda+z_j p'(z_j)=0},\vspace{1.8mm}\\
\mathlarger{2[p(z_j)-\lambda]+z_j p'(z_j)=-\frac{1}{z_j}}.
\end{cases}$$
Thus we can write $$f(z)=(z-z_j)^{n_j}g_j(z)$$ for some integer $n_j>1$ and $g_j\in L^2_a$ with $g(z_j)\neq 0$. According to the above expression of
$f$, we get
$$-\frac{f'(z)}{f(z)}=\frac{-n_j}{z-z_j}+\frac{-g_j^{\prime}(z)}{g_j(z)},$$
which gives that $-f'/f$ has a pole $z_j $ with order at most $1$. On the other hand, note that the function $$\frac{2[p(z)-\lambda]+zp'(z)}{1+z[p(z)-\lambda]}$$ has a pole $z_j $ with order greater than $1$, since $z_j$ is  a  multiple zero of $1+z[p(z)-\lambda]$ (by the assumption) but not a zero of $2[p(z)-\lambda]+z p'(z)$. This implies that $\lambda$ can not satisfy  the following equation:
$$ -\frac{f'(z)}{f(z)}= \frac{2[p(z)-\lambda]+zp'(z)}{1+z[p(z)-\lambda]},$$
which contradicts that $\lambda$ is an eigenvalue of $T_{\overline{z}+p}$. Thus $z_1, z_2, \cdots , z_k$  are all simple zeros of $1+z[p(z)-\lambda]$ in the unit disk.

In order to derive the condition for $z_j$ in (\ref{eqonp}), we first  consider the case of
$$2[p(z_j)-\lambda]+z_j p'(z_j)\neq 0,$$
where $z_j\in \big\{z_1, z_2, \cdots , z_k\big\}$.
Since the function $f$ is analytic on $\mathbb D$, we can choose a small circle $\gamma\subset \mathbb D$ centered at $z_j$ such that there is no zero of $f$ on $\gamma$ and all other zeros of $1+z[p(z)-\lambda]$ are outside $\gamma$. Then the Cauchy residue theorem
implies that there exists some  $n_j\in \big\{0, 1, 2, \cdots \big\}$ such that
\begin{align*}
-n_j&=-\frac{1}{2\pi \mathrm{i}}\int_{\gamma} \frac{f'(z)}{f(z)}dz\\
&=\frac{1}{2\pi  \mathrm{i}}\int_{\gamma} \frac{2[p(z)-\lambda]+zp'(z)}{1+z[p(z)-\lambda]}dz \ \ \ \ \ \  \ \ \  \big(\mathrm{by\ Equation} \ (\ref{difff})\big)\\
&=\mathrm{Res}\bigg(\frac{2[p(z)-\lambda]+zp'(z)}{1+z[p(z)-\lambda]}; \  z_j\bigg)\\
&=\frac{2p(z_j)-2\lambda+z_j p'(z_j)}{p(z_j)-\lambda+z_j p'(z_j)}\\
&=1+\frac{1}{1+\frac{z_j}{p(z_j)-\lambda}p'(z_j)}\\
&=1+\frac{1}{1-z_j^2 p'(z_j)},
\end{align*}
where the fourth equality follows from that $2[p(z_j)-\lambda]+z_j p'(z_j)\neq 0$ and $z_j$ is a simple zero of $1+z[p(z)-\lambda]$, and
the last equality comes from that $$p(z_j)-\lambda=-\frac{1}{z_j}.$$ Solving the above equation gives
$$z_j^2 p'(z_j)=\frac{n_j+2}{n_j+1},$$
to obtain (\ref{eqonp}).

For the case of   $$2[p(z_j)-\lambda]+z_j p'(z_j)=0, $$ we have following system:
$$\begin{cases}
\mathlarger {2[p(z_j)-\lambda]+z_j p'(z_j)=0},\vspace{2mm}\\
\mathlarger {1+z_j[p(z_j)-\lambda]=0},
\end{cases}$$
which implies that
$$z_j^2 p'(z_j)=2.$$
 In conclusion, we obtain (\ref{eqonp}) for $n_j=0$, to finish the proof of the  necessity.

Now we turn to the proof of the  sufficiency, we need to show $\lambda\in \sigma_p(T_{\overline{z}+p})$. By the assumption, we  need to consider two cases. First, let us  discuss the case that $1+z[p(z)-\lambda]$ has no zeros on the closed unit disk $\overline{\mathbb D}$. In this case, we define a nonzero function $f$ by
$$f(z)=\exp\bigg\{-\int_0^z \frac{2[p(w)-\lambda]+wp'(w)}{1+w[p(w)-\lambda]} dw\bigg\},$$
  where the integrand
 $$\frac{2[p(w)-\lambda]+wp'(w)}{1+w[p(w)-\lambda]}$$ is a bounded  and  analytic  function on $\mathbb D$. Thus $f$  belongs to $L_a^2$. Moreover,
 simple calculations give us that
 \begin{align*}
 \frac{f'(z)}{f(z)}
&=-\frac{2[p(z)-\lambda]+zp'(z)}{1+z[p(z)-\lambda]},
\end{align*}
 which means that the function $f$ defined above  satisfies Equation (\ref{difff}), and so $\lambda$ is an eigenvalue of the Toeplitz  operator $T_{\overline{z}+p}$.

To finish the proof, it remains to consider the case that $1+z[p(z)-\lambda]$ has simple zeros  $\big\{z_{1}, z_2, \cdots, z_k\big\}$  in $\mathbb D$ such that
$$z_j^2 p'(z_j)=\frac{n_j+2}{n_j+1}$$
for some integer $n_j\in \big\{0, 1, 2, \cdots\big\}$ with $j=1, 2, \cdots, k$. In order to show that $\lambda$ is an eigenvalue of $T_{\overline{z}+p}$  in this case, we need to solve  Equation (\ref{difff}).  Let us consider the following function:
$$f(z)=g(z)\prod_{j=1}^k (z-z_j)^{n_j},$$
where $$g(z)=\exp\bigg\{-\int_0^z \bigg(\frac{2[p(w)-\lambda]+w p'(w)}{1+w[p(w)-\lambda]}+\sum_{j=1}^k\frac{n_j}{w-z_j}\bigg)dw \bigg\}.$$
We are going to verify that $f$ is an eigenvector of $T_{\overline{z}+p}$.

Since the zeros $\big\{z_{1}, z_2, \cdots, z_k\big\}\subset \mathbb D$ of $1+w[p(w)-\lambda]$ are all simple, we can write
$$\frac{2[p(w)-\lambda]+wp'(w)}{1+w[p(w)-\lambda]}=\sum_{j=1}^{k}\frac{a_j}{w-z_j}+\sum_{l=1}^{N}\frac{b_l}{(w-w_l)^{m_l}}$$
for some $w_l$ outside the closure of the unit disk, $m_l\in \mathbb N$ and $a_j, b_l$ are all complex  constants.
Recall that $$z_j^2 p'(z_j)=\frac{n_j+2}{n_j+1} \ \ \ \ \ \ \big(j=1,2,\cdots, k\big)$$  and  each $z_j \in \mathbb D$ is the zero  of $1+z[p(z)-\lambda]$ with multiplicity $1$. Repeating the same arguments as the one in the proof of the necessity, we obtain by the Cauchy residue theorem that
 $a_j=-n_j$ for every $1\leqslant j \leqslant k$, so $$\frac{2[p(w)-\lambda]+w p'(w)}{1+w[p(w)-\lambda]}+\sum_{j=1}^k\frac{n_j}{w-z_j}=\sum_{l=1}^{N}\frac{b_l}{(w-w_l)^{m_l}}$$
 is in $H^\infty$ and hence
$g\in H^\infty$. Moreover, from the definitions of $f$ and $g$ we have
\begin{align*}
\frac{f'(z)}{f(z)}&=\sum_{j=1}^k\frac{n_j}{z-z_j}+\frac{g'(z)}{g(z)}\\
&=\sum_{j=1}^k\frac{n_j}{z-z_j}-\bigg(\frac{2[p(z)-\lambda]+z p'(z)}{1+z[p(z)-\lambda]}+\sum_{j=1}^k\frac{n_j}{z-z_j}\bigg)\\
&=-\frac{2[p(z)-\lambda]+z p'(z)}{1+z[p(z)-\lambda]}.
\end{align*}
Noting that $f$ is analytic on $\mathbb D$ and bounded on $\overline{\mathbb D}$, we conclude that $f$ is a nonzero solution of  Equation (\ref{difff}) in the Bergman space  $L_a^2$. Hence $f$ is an eigenvector of $T_{\overline{z}+p}$ corresponding to the eigenvalue $\lambda$. This completes the proof of Theorem \ref{key lemma}.
\end{proof}

 The above theorem leads to  the following complete characterization on the invertibility of the Toeplitz operator $T_{\overline{z}+p}$ with $p\in H^\infty \cap C(\overline{\mathbb D})$  immediately.
\begin{thm}\label{invertibility}
Let $p$ be a function in $H^\infty\cap C(\overline{\mathbb D})$. Then the Toeplitz operator $T_{\overline{z}+p}$  is invertible on the Bergman space $L_a^2$ if and only if the following two conditions hold: \\
$\mathrm{(i)}$ $1+zp$ has no zeros  on the unit circle $\partial \mathbb D$;\\
$\mathrm{(ii)}$ $1+zp$ has exactly one simple zero $z_0$ in the open disk $\mathbb D$ which satisfies that
$$z_0^2p'(z_0)-\frac{n+2}{n+1}\neq 0$$
for any nonnegative integer $n$.
\end{thm}

\section{Quadratic polynomials}
In this section, we  show that the spectrum of the Topelitz operator $T_{\overline{z}+p}$ is connected for every quadratic polynomial $p$. More precisely,  we obtain the following theorem which is analogous to  \cite[Theorem 4.6.1]{Arv} or \cite[Corollary 7.28]{Dou}  about the spectra of  Toeplitz operators on the Hardy space.

\begin{thm}\label{(z)+az2+bz+c}
Let $\varphi(z)=\overline{z}+p(z)$, where $p$ is a quadratic polynomial. The spectrum of the Toeplitz operator operator $T_{\varphi}$ is given by
$$\sigma(T_\varphi)=\varphi(\partial \mathbb D)\bigcup \Big\{\lambda\in \mathbb C: \lambda\notin \varphi(\partial \mathbb D) \ \mathrm{and}\ \mathrm{wind}\big(\varphi(\partial \mathbb D), \lambda\big)\neq 0\Big\},$$
which coincides with the spectrum of the corresponding Hardy-Toeplitz operator with symbol $e^{-\mathrm{i}\theta}+p(e^{\mathrm{i}\theta})$. Hence the spectrum of $T_{\overline{z}+p}$ is connected for every $p(z)=az^2+bz+c$ with $a, b, c \in \mathbb C$.
\end{thm}

\begin{proof} Let $p$ be a quadratic polynomial $az^2+bz+c$ and let $\varphi(z)=\overline{z}+p(z)$. By Lemma \ref{Fredholm index},  we have
\begin{align*}
 \sigma(T_\varphi)&=\varphi(\partial \mathbb D)\bigcup \Big\{\lambda\in \mathbb C: \lambda\notin \varphi(\partial \mathbb D) \ \mathrm{and}\ \mathrm{wind}\big(\varphi(\partial \mathbb D), \lambda\big)\neq 0\Big\}\\
 & \ \ \ \ \bigcup \bigg(\sigma_p (T_\varphi)\bigcap \Big\{\lambda \in \mathbb C:  \lambda \notin \sigma_e(T_\varphi)\ \mathrm{and}\ \mathrm{index}(T_{\varphi}-\lambda I)=0\Big\}\bigg).
 \end{align*}
 To prove the theorem we need only to show
  that  $$ \sigma_p (T_\varphi)\bigcap \Big\{\lambda \in \mathbb C:  \lambda \notin \sigma_e(T_\varphi)\ \mathrm{and}\ \mathrm{index}(T_{\varphi}-\lambda I)=0\Big\}=\varnothing .$$
Since $\sigma_e(T_\varphi)=\varphi(\partial \mathbb D)$,  and for each complex number $\mu$,
 $$T_{\phi}-\mu I=T_{\phi -\mu}=T_{\overline{z}+az^2+bz+(c-\mu )},$$
 it is sufficient to show that $0$ does not belong to
$$\sigma_p (T_\varphi)\bigcap \Big\{\lambda\in \mathbb C: \lambda \notin \varphi(\partial \mathbb D)\ \mathrm{and}\ \mathrm{wind}\big(\varphi(\partial \mathbb D), \lambda\big)=0\Big\}$$ for any $\varphi(z)=\overline{z}+(az^2+bz+c)$.

To do so, suppose that $0$ belongs to
$$\sigma_p (T_\varphi)\bigcap \Big\{\lambda\in \mathbb C: \lambda \notin \varphi(\partial \mathbb D)\ \mathrm{and}\ \mathrm{wind}\big(\varphi(\partial \mathbb D), \lambda\big)=0\Big\}$$ for some $\varphi(z)=\overline{z}+(az^2+bz+c).$
Noting that $\mathrm{index}(T_{\varphi})=0$, Lemma \ref{Fredholm index} gives us
$$0=\mathrm{index}(T_\varphi)=-\mathrm{wind}\big(\varphi(\partial \mathbb D), 0\big)=-\mathrm{wind}\bigg(\frac{az^3+bz^2+cz+1}{z}\bigg|_{\partial \mathbb D}, \ 0\bigg).$$
By the argument principle, the cubic polynomial
 $az^3+bz^2+cz+1$ has exactly one zero in the disk $\mathbb D$ with multiplicity 1 and hence  has the following factorization:
 $$az^3+bz^2+cz+1=a(z-\alpha)(z+\beta)(z+\gamma)$$
with $|\alpha|<1$, $|\beta|>1$ and $|\gamma|>1$. Evaluating both sides of the above equality at $0$  gives
\begin{equation}\label{aabg}
a\alpha \beta \gamma=-1.
\end{equation}

Since $0$ is  an eigenvalue of the Toeplitz operator $T_{\varphi}$, $T_{\varphi}$ has a nonzero eigenvector   $f$ in $ L_a^2$ such that
$$T_{\varphi}f=T_{\overline{z}+p}f=0.$$   By Theorem \ref{key lemma}, we get
\begin{equation}\label{abg}
\alpha^2p^{\prime}(\alpha ) =\frac{n+2}{n+1}
\end{equation}
for some $n\in \mathbb N$.
Since $$zp(z)+1=az^3+bz^2+cz+1=a(z-\alpha)(z+\beta)(z+\gamma),$$
taking derivatives both sides of the above equality and evaluating at $\alpha$ give
$$\alpha p(\alpha)+\alpha^2p^{\prime}(\alpha )=a\alpha(\alpha+\beta)(\alpha+\gamma)=a\alpha \beta \gamma\Big(1+\frac{\alpha}{\beta}\Big)\Big(1+\frac{\alpha}{ \gamma}\Big).$$
Note that $$\alpha p(\alpha)+1=0,$$ thus combining (\ref{aabg}) with (\ref{abg})  gives
\begin{align}\label{negative}
\Big(1+\frac{\alpha}{\beta}\Big)\Big(1+\frac{\alpha}{ \gamma}\Big)= -\frac{1}{n+1}<0
\end{align}
for some $n\in \mathbb N$. This will lead a contradiction by
 the following simple observation:
\begin{align}\label{nnegative}
zw<0 \Longrightarrow \mathrm{Re}(z)\cdot \mathrm{Re}(w) \leqslant 0.
\end{align}
To derive a contradiction,  using the above observation and Inequality (\ref{negative})  we have
\begin{equation}\label{newneg}
\mathrm{Re}\Big(1+\frac{\alpha}{\beta}\Big)\cdot \mathrm{Re}\Big(1+\frac{\alpha}{\gamma}\Big)\leqslant 0.
\end{equation}
On the other hand, we have that  $\frac{\alpha}{\beta}$ and   $\frac{\alpha}{\gamma}$ are in the unit disk, to get
$$\mathrm{Re}\Big(1+\frac{\alpha}{\beta}\Big)>0 \ \ \ \mathrm{and}\ \ \  \mathrm{Re}\Big(1+\frac{\alpha}{\gamma}\Big)> 0.$$
This contradicts with (\ref{newneg}) and yields that $0$ does not belong to
$$\sigma_p (T_\varphi)\bigcap \Big\{\lambda\in \mathbb C: \lambda \notin \varphi(\partial \mathbb D)\ \mathrm{and}\ \mathrm{wind}\big(\varphi(\partial \mathbb D), \lambda\big)=0\Big\}$$
for any $\varphi(z)=\overline{z}+(az^2+bz+c)$.

To complete the proof, we need to prove Observation $(\ref{nnegative})$. Writing
$$z=x_1+\mathrm{i}y_1 \  \ \mathrm{and} \  \ w=x_2+\mathrm{i}y_2 \ \   \    \ \big(x_1, x_2, y_1, y_2\in \mathbb R\big),$$
we have that  $zw<0$ if and only if
$$\begin{cases}
\mathlarger{x_1x_2-y_1 y_2<0},\vspace{2mm}\\
\mathlarger{x_1y_2+x_2y_1=0}.
\end{cases}$$
If $zw<0$ and neither $x_1$ nor $x_2$ equals $0$, then
 $x_1$ and $x_2$ have opposite signs. Otherwise,
 if both $x_1$ and $x_2$ have the same signs,  then the product $x_1x_2$ is positive. Thus the first inequality  above gives
 $$y_1 y_2>x_1 x_2>0,$$  so $y_1$ and $y_2$ also have the same signs. But the second equality  above  gives
$$\frac{x_1}{x_2}=-\frac{y_1}{y_2}.$$
The left-hand side of the above equality is positive but the right-hand side is negative. We get
 a contradiction to complete the proof of Theorem \ref{(z)+az2+bz+c}.
\end{proof}

It was shown in \cite[Theorem 3.1]{ZZ} that $\sigma(T_{h})=\mathrm{clos}[h(\mathbb D)]$ is a solid ellipse if $h(z)=\overline{z}+(az+b) \ (a, b\in \mathbb C)$. Thus, Theorem  \ref{(z)+az2+bz+c} is true for any polynomial $p$ with degree at most $2$. From the proof of Theorem  \ref{(z)+az2+bz+c}, we get the following two corollaries immediately.

\begin{cor}\label{invertibility(z)+az2+bz+c}
Let $\varphi(z)=\overline{z}+(az^2+bz+c)$, where $a, b$ and  $c$ are all complex constants.  Then the Toeplitz operator $T_\varphi$ is invertible if and only if the cubic polynomial
$az^3+bz^2+cz+1$
has a unique zero in the unit disk $\mathbb D$ with multiplicity 1,  but does not have any zero on the unit circle $\partial \mathbb D$.
\end{cor}
\begin{proof}
If $T_\varphi$ is invertible, then $T_\varphi$ is a Fredholm operator with Fredholm  index $0$. By Lemma \ref{Fredholm index}, since $0$ is not in  $ \varphi(\partial \mathbb D)$ and
$$0=-\mathrm{wind}\big(\varphi(\partial \mathbb D), 0\big)=-\mathrm{wind}\bigg(\frac{az^3+bz^2+cz+1}{z}\bigg|_{\partial \mathbb D},\  0\bigg),$$
  the argument principle implies that $az^3+bz^2+cz+1$ has a unique zero with multiplicity $1$ in the unit disk $\mathbb D$.

Conversely, we assume that the cubic polynomial satisfies the conditions in the corollary, then  $T_\varphi$ is  a Fredholm operator with $\mathrm{index}(T_\varphi)=0$.
We are going to show that $T_\varphi$ is invertible on the Bergman space. If it is not invertible, then there exists a nonzero function $f\in L_a^2$ such that
$T_\varphi f=0$.

Write  $$az^3+bz^2+cz+1=a(z-\alpha)(z+\beta)(z+\gamma),$$
where $|\alpha|<1$, $|\beta|>1$ and  $|\gamma|>1$. Since $f$ is nonzero, the proof of Theorem \ref{(z)+az2+bz+c} shows
$$\mathrm{Re}\Big(1+\frac{\alpha}{\beta}\Big)\cdot \mathrm{Re}\Big(1+\frac{\alpha}{\gamma}\Big)\leqslant 0.$$
 However, Observation $(\ref{nnegative})$ in the proof of Theorem \ref{(z)+az2+bz+c} tells us that  one  of $\frac{\alpha}{\beta}$ and $ \frac{\alpha}{\gamma}$ lies outside  the unit disk. This implies that either $\beta$ or $\gamma$ is inside the unit disk, since $|\alpha|<1$. It is a contradiction.  This finishes the proof of Corollary \ref{invertibility(z)+az2+bz+c}.
\end{proof}

\begin{cor}
Let $\varphi(z)=\overline{z}+(az^2+bz+c)$ with $a, b, c \in \mathbb C$. Then we have
$$\sigma(T_\varphi)\subset \mathrm{clos}[\varphi(\mathbb D)].$$
\end{cor}

\begin{proof}
 Let $\lambda$ be any complex number not in $\mathrm{clos}[\varphi(\mathbb D)]$. Thus we obtain
 $$|\varphi (z)-\lambda |\geqslant \delta >0$$
 for some constant $\delta$ and all $z$ in the unit disk.

Let $$F(e^{\mathrm{i}\theta}, t):=\varphi(te^{\mathrm{i}\theta})-\lambda,  \ \ \ \theta\in [0, 2\pi], \ \ \ t\in [0, 1]. $$
Then we have
$F(e^{\mathrm{i}\theta}, t)\in C\big(\partial\mathbb D \times [0, 1], \mathbb C \backslash\{0\}\big)$ and
$$F(e^{\mathrm{i}\theta}, 0)=\varphi(0)-\lambda, \ \ \  F(e^{\mathrm{i}\theta}, 1)=\varphi(e^{\mathrm{i}\theta})-\lambda, \ \ \  e^{\mathrm{i}\theta}\in  \partial\mathbb D.$$
This shows that the curve $\varphi(\partial \mathbb D)-\lambda$ is  homotopy equivalent to a nonzero point  $\varphi(0)-\lambda $   \big(in $\mathbb C\backslash\{0\}$\big). Thus  we have
$$0=\mathrm{wind}\big(\varphi(\partial \mathbb D)-\lambda, 0\big)=\mathrm{wind}\bigg(\frac{az^3+bz^2+(c-\lambda )z+1}{z}\bigg|_{\partial \mathbb D}, \ 0\bigg).$$
The argument principle gives that the polynomial $$az^3+bz^2+(c-\lambda )z+1$$ has exactly one zero with multiplicity $1$ in $\mathbb D$. By  Corollary \ref{invertibility(z)+az2+bz+c}, we have that $T_{\varphi-\lambda}$ is invertible, and hence $\lambda$ is not in the spectrum $\sigma (T_{\varphi})$. So we obtain
$$\sigma (T_{\varphi}) \subset \mathrm{clos}[\varphi(\mathbb D)],$$
to complete the proof of this corollary.
\end{proof}

\begin{rem} Athough it is true \cite[Theorem 3.1] {ZZ} that for every linear polynomial $p$ and $\varphi(z) =\overline{z}+p(z)$, we have
$$\sigma (T_{\varphi})=\mathrm{clos}[\varphi(\mathbb D)],$$
 the inclusion
 $$\mathrm{clos}[\varphi(\mathbb D)]\subset \sigma(T_\varphi)$$
  does not hold for general harmonic polynomials $$\varphi(z)=\overline{z}+(az^2+bz+c).$$ Indeed, it was shown in   \cite[Theorem 4.1]{ZZ} that the Toeplitz operator $T_{\overline{z}+(z^2-z)}$ is invertible on the Bergman space $L_a^2$, but clearly the harmonic function $\overline{z}+(z^2-z)$ is not invertible in $L^\infty(\mathbb D)$.
\end{rem}

\section{Polynomials with higher degree}

In this section, we mainly study the spectral structure of the Toeplitz operator $T_{\overline{z}+p}$, where $p$ is a polynomial with degree greater than $2$. For each integer $k>2$, we will construct a polynomial $p$ of degree $k$ such that the spectrum of the  Toeplitz operator $T_{\overline{z}+p}$ has at least one isolated point but has at most finitely many isolated points.

The first main result of this section is contained in the following theorem, which provides a general method  to construct Toeplitz operators with harmonic symbols on the Bergman space having disconnected spectra.

\begin{thm}\label{deg>2}
For each integer $k$ greater than $2$, there is a polynomial $p$ with degree $k$ such that the spectrum  $\sigma(T_{\overline{z}+p})$ has an isolated point and hence is disconnected.
\end{thm}
\begin{proof}
Let $k$ and $n$ be two integers such that $k\geqslant 3$ and $n \geqslant 1$. Let $\beta$ be the complex number defined by: $$\beta=\bigg[1-\Big(\frac{1}{n+1}\Big)^{\frac{1}{k}}e^{\frac{\pi}{k}\mathrm{i}}\bigg]^{-\frac{1}{k+1}}.$$ First we  show $|\beta|>1$. To do so, using a simple computation we have
\begin{align*}
\bigg|1-\Big(\frac{1}{n+1}\Big)^{\frac{1}{k}}e^{\frac{\pi}{k}\mathrm{i}}\bigg|^2 &=1-2\Big(\frac{1}{n+1}\Big)^{\frac{1}{k}} \cos\Big(\frac{\pi}{k}\Big)+\Big(\frac{1}{n+1}\Big)^\frac{2}{k}\\
&=1-\Big[2^k\Big(\frac{1}{n+1}\Big) \cos^k\Big(\frac{\pi}{k}\Big)\Big]^{\frac{1}{k}}+\Big(\frac{1}{n+1}\Big)^\frac{2}{k}\\
&\leqslant  1-\Big[2^k\Big(\frac{1}{n+1}\Big) \cos^k\Big(\frac{\pi}{3}\Big)\Big]^{\frac{1}{k}}+\Big(\frac{1}{n+1}\Big)^\frac{2}{k}\\
&=1-\Big(\frac{1}{n+1}\Big)^{\frac{1}{k}}+\Big(\frac{1}{n+1}\Big)^\frac{2}{k}\\
&<1-\Big(\frac{1}{n+1}\Big)^{\frac{1}{k}}+\Big(\frac{1}{n+1}\Big)^\frac{1}{k}=1
\end{align*}
for $k\geqslant 3$ and $n\geqslant 1$. The first inequality follows from that the function $\cos(t)$ is decreasing for positive $t$ nearby $0$.

Letting $\alpha=-\beta^{-k}$, then we have  $|\alpha|<1$.  Then we can define a class of polynomials $p$ such that $\sigma(T_{\overline{z}+p})$ is disconnected as follows:
$$p(z)=\frac{(z-\alpha)(z+\beta)^k-1}{z}.$$
Clearly, the degree of the polynomial $p$ is $k$.
We will show that $0$ is an isolated point in  $\sigma(T_{\overline{z}+p})$. To do so, we first show that $0$ is an eigenvalue of $T_{\overline{z}+p}$.  As in the proof of Theorem \ref{key lemma}, simple
calculations yield
$$\alpha^2 p'(\alpha)=\frac{n+2}{n+1}.$$
 By Theorem \ref{key lemma}, we have that $0$ is an eigenvalue  of $T_{\overline{z}+p}$.

Next we will show that $0$ is  an isolated point of $\sigma(T_{\overline{z}+p})$. Noting that
$$\overline{z}+p(z)=\frac{1+zp(z)}{z}=\frac{(z-\alpha)(z+\beta)^k}{z}$$
on the unit  circle, we have that  $0$ is not in $ \sigma_e(T_{\overline{z}+p})$. Since  $\alpha$ is  the unique zero of $1+zp(z)$ in the unit disk,  we have $$\mathrm{index}(T_{\overline{z}+p})=0,$$ and hence $0$ is in the intersection
\begin{align*}
\Lambda:=\sigma_p (T_{\overline{z}+p})\bigcap \Big\{\lambda \in \mathbb C: \lambda \notin \sigma_e(T_{\overline{z}+p})\ \mathrm{and}\ \mathrm{index}(T_{\overline{z}+p}-\lambda I)=0\Big\}.
\end{align*}

Suppose that $\lambda$ belongs to $\Lambda$, then we have by the argument principle that there exists a $z_\lambda\in \mathbb D$ such that
$$1+z_\lambda[p(z_\lambda)-\lambda]=0.$$
Moreover, since $\lambda$ is an eigenvalue of $T_{\overline{z}+p}$, the proof of Theorem \ref{key lemma} implies that
$$z_\lambda^2p'(z_\lambda)=\frac{n+2}{n+1}$$
for some integer $n\geqslant 0$. Therefore, $\Lambda$ is contained  in the following countable set:
$$ \bigg\{\lambda \in \mathbb C: z_\lambda \ \text{is\ the \ root \ of}\   \ \lambda=\frac{1}{z}+p(z)\ \text{in}\  \mathbb D \ \ \mathrm{and}\ \  z_\lambda^2p'(z_\lambda)=\frac{n+2}{n+1} \mathrm{\ for\  some\ } n\in \mathbb N \bigg\},$$
which is a discrete set and hence is not connected.
 By Theorem \ref{spectral picture theorem}, we conclude that the eigenvalue $0$ is an isolated point in $\sigma(T_{\overline{z}+p})$. Consequently, the spectrum of  the Toeplitz operator $T_{\overline{z}+p}$ is disconnected, to complete the proof of Theorem \ref{deg>2}.
\end{proof}

An interesting  problem in the study of the Topelitz operator theory is to characterize the spectral structure of such operators. However, as mentioned in Section 1, people know very little about this problem. Although we have shown in the previous theorem that there exists a class of Toeliptz operators with harmonic polynomial symbols such that they have  disconnected spectra, it is natural to study more about the topological structure for the spectra of this class of Toeplitz operators. As the isolated points in the spectrum of an operator can be used for producing hyperinvariant subspaces (which was mentioned in Section 1) and for studying Weyl type theorems, we hope to further discuss the isolated points in the spectrum of the Toeplitz operator constructed in Theorem \ref{deg>2}.

 In fact, we will show in the next theorem that there are at most finitely  many isolated points in $\sigma(T_{\overline{z}+p})$ for each polynomial $p$ defined in the above theorem, which suggests that there can not be too many isolated points in the spectra of such class of operators. For this purpose, we need to estimate the number of zeros of certain  high-degree polynomials. In order to analyze the zero distribution of complex polynomials, we need the following lemma, which is known as ``the continuous dependence of the zeros of a polynomial on its coefficients", see \cite[appendix A]{Ost} or \cite[Theorem 1.3.1]{RS} if necessary.

\begin{lemma}\label{continuous dependence}
Let $$p(z)=z^n+a_{n-1}z^{n-1}+\cdots+a_1 z+a_0$$
be a monic polynomial with complex coefficients. Then, for every small $\epsilon>0$, there exists a $\tau>0$ such that for any polynomial
$$q(z)=z^n+b_{n-1}z^{n-1}+\cdots+b_1z+b_0$$
satisfying $$\max\limits_{0\leqslant k\leqslant n-1}|a_k-b_k|<\tau,$$
the zeros $z_j$ and $w_j$ of $p$ and $q$, respectively, can be ordered in such a way that
we have $$|z_j-w_j|<\epsilon  \ \ \  \big(j=1, 2, \cdots, n\big).$$
\end{lemma}

Another main result of this section  is contained in the next theorem. For convenience, we only consider the case of $n=1$ in Theorem \ref{deg>2}.

\begin{thm}\label{deg=k}
Fix an integer $k\geqslant 3$. Suppose that $\beta$ is such that
$$\beta^{k+1}=\frac{1}{1-(\frac{1}{2})^{\frac{1}{k}}e^{\frac{\pi}{k}\mathrm{i}}},$$
and let $\alpha =-\beta^{-k}.$
Let $p$  be the polynomial with degree $k$ as the following:
$$p(z)=\frac{(z-\alpha)(z+\beta)^k-1}{z}.$$
Then $0$ is an isolated point in  $\sigma(T_{\overline{z}+p})$ and $\sigma(T_{\overline{z}+p})$ has at most finitely many isolated points.
 \end{thm}

In order to prove  Theorem \ref{deg=k}, we need one more lemma.
\begin{lemma}\label{at least two zeros}
 Let $p$ be the polynomial with degree $k\geqslant 3$ defined in Theorem \ref{deg=k}. There is an $N\in \mathbb N$ such that for each integer $n>N$ and for each solution $w$ (in the disk $\mathbb D$) of the following equation:
  \begin{equation*}
  z^2p'(z)=\frac{n+2}{n+1},
  \end{equation*}
 setting $$\lambda =\frac{1}{w}+p(w), $$ the polynomial $1+z[p(z)-\lambda]$ has at least two zeros  (counting multiplicities)
  in the open unit disk.
  \end{lemma}

\begin{proof} The proof of this lemma will be done in several steps. First we show that the equation $z^2p'(z)=1$ has exactly two distinct  roots in the  open unit disk $\mathbb D$. From the assumption, we have by the definition of $p$ that $$1+zp(z)=(z-\alpha)(z+\beta)^k.$$ Taking derivatives of both sides of the above equality and then multiplying both sides by $z$ give  us that
$$zp(z)+z^2p'(z)=z(z+\beta)^k+kz(z-\alpha)(z+\beta)^{k-1}.$$
This implies
\begin{align*}
z^2p'(z)-1&=z^2p'(z)-(z-\alpha)(z+\beta)^k+zp(z)\\
&=z(z+\beta)^k+kz(z-\alpha)(z+\beta)^{k-1}-(z-\alpha)(z+\beta)^k\\
&=(z+\beta)^{k-1}\big[kz^2+\alpha(1-k)z+\alpha \beta\big]\\
&=(z+\beta)^{k-1}\Big(kz^2+\frac{k-1}{\beta^{k}}z-\frac{1}{\beta^{k-1}}\Big),
\end{align*}
 where the last equality follows from $\alpha=-\beta^{-k}$. Thus  $z^2p'(z)=1$ has a  multiple root $-\beta$ and other two complex roots  are given by
  $$z_{\infty}=-\frac{\sqrt{(k-1)^2+4k\beta^{k+1}}+(k-1)}{2k\beta^k}$$
  and
 $$w_{\infty}=\frac{\sqrt{(k-1)^2+4k\beta^{k+1}}-(k-1)}{2k\beta^k}.$$
 Clearly, $z_\infty$ is not equal to $w_\infty$.

 Before going further, we first show that $z_\infty$ and $w_\infty$ are both in $\mathbb D$. To do so, we need some routine computations.
Noting that
  \begin{equation}\label{beta1}
  \begin{split}
\frac{1}{\beta^{k+1}}&=1-\Big(\frac{1}{2}\Big)^{\frac{1}{k}}e^{\frac{\pi}{k}\mathrm{i}}\\
&=1-\Big(\frac{1}{2}\Big)^{\frac{1}{k}}\cos\Big(\frac{\pi}{k}\Big)-\mathrm{i}\Big(\frac{1}{2}\Big)^{\frac{1}{k}}\sin\Big(\frac{\pi}{k}\Big),
\end{split}
\end{equation}
we obtain
\begin{equation}\label{beta2}
  \begin{split}
\frac{1}{|\beta|^k}&=\Big|1-\Big(\frac{1}{2}\Big)^{\frac{1}{k}}e^{\frac{\pi}{k}\mathrm{i}}\Big|^{\frac{k}{k+1}}=\bigg(\Big|1-\Big(\frac{1}{2}\Big)^{\frac{1}{k}}e^{\frac{\pi}{k}\mathrm{i}}\Big|^2\bigg)^{\frac{k}{2(k+1)}}\\
&=\bigg[1-2^{1-k}\cos\Big(\frac{\pi}{k}\Big)+4^{-k}\bigg]^{\frac{k}{2(k+1)}}
  \end{split}
\end{equation}
for $k\geqslant 3$. In order to estimate $|z_\infty|$, $|w_\infty|$ and $|\beta|^{-k}$,  we need  the following two real-valued functions on the interval  $\big(0,\frac{1}{3}\big]$:
$$A(x)=\frac{1-2^{-x}\cos(\pi x)}{1-2^{1-x}\cos(\pi x)+4^{-x}}-\Big(x+\frac{1}{20x}\Big)$$
and
$$B(x)=\Big[1-2^{1-x}\cos(\pi x)+4^{-x}\Big]\Big(2x+\frac{5}{16x}\Big)^2.$$
Using Taylor's series for $\cos(t)$ with $t\in \big(0, \frac{\pi}{3}\big]$, we can show by standard calculus  that
$$A(x)\geqslant \frac{1}{10} \ \ \ \ \ \ \mathrm{and} \ \ \ \ \ \  1\leqslant B(x) \leqslant \frac{5}{2}$$
for all $0<x\leqslant \frac{1}{3}$. This yields
\begin{align}\label{k1}
\frac{1-(\frac{1}{2})^{\frac{1}{k}}\cos(\frac{\pi}{k})}{\big|1-\big(\frac{1}{2}\big)^{\frac{1}{k}}e^{\frac{\pi}{k}\mathrm{i}}\big|^2}\geqslant \frac{k}{20}+\frac{1}{k}
\end{align}
and
\begin{align}\label{k2}
\frac{2}{5}\Big(\frac{5k}{16}+\frac{2}{k}\Big)^2\leqslant \Big|1-\big(\frac{1}{2}\big)^{\frac{1}{k}}e^{\frac{\pi}{k}\mathrm{i}}\Big|^{-2}\leqslant \Big(\frac{5k}{16}+\frac{2}{k}\Big)^2
\end{align}
for all $k\geqslant 3$. Combining  (\ref{beta1}) with (\ref{k1}) and (\ref{k2})  gives
\begin{equation}\label{estimate1}
\begin{split}
\bigg|\Big(\frac{k-1}{k}\Big)^2+\frac{4\beta^{k+1}}{k}\bigg|^{\frac{1}{2}}&=\mathlarger{\bigg|\bigg(\frac{k-1}{k}\bigg)^2+\frac{4}{k}\bigg( \frac{1-(\frac{1}{2})^{\frac{1}{k}}\cos(\frac{\pi}{k})+\mathrm{i}(\frac{1}{2})^{\frac{1}{k}}\sin(\frac{\pi}{k})}{\big|1-\big(\frac{1}{2}\big)^{\frac{1}{k}}e^{\frac{\pi}{k}\mathrm{i}}\big|^2}\bigg)\bigg|^{\frac{1}{2}}}\\
&\geqslant \sqrt{\bigg(\frac{k-1}{k}\bigg)^2+\frac{4}{k}\left( \frac{1-(\frac{1}{2})^{\frac{1}{k}}\cos(\frac{\pi}{k})}{\big|1-\big(\frac{1}{2}\big)^{\frac{1}{k}}e^{\frac{\pi}{k}\mathrm{i}}\big|^2}\right)}\\
&\geqslant \sqrt{\bigg(\frac{k-1}{k}\bigg)^2+\frac{4}{k}\bigg(\frac{k}{20}+\frac{1}{k}\bigg)}\\
&=\sqrt{1+\frac{(k-5)^2}{5k^2}}\geqslant 1
\end{split}
\end{equation}
for all $k\geqslant 3$ and
\begin{equation}\label{estimate2}
\begin{split}
\bigg|\Big(\frac{k-1}{k}\Big)^2+\frac{4\beta^{k+1}}{k}\bigg|^{\frac{1}{2}}&=\mathlarger{\left|\Big(\frac{k-1}{k}\Big)^2+\frac{4}{k}\left( \frac{1}{1-\big(\frac{1}{2}\big)^{\frac{1}{k}}e^{\frac{\pi}{k}\mathrm{i}}}\right)\right|^{\frac{1}{2}}}\\
&\leqslant \sqrt{\Big(\frac{k-1}{k}\Big)^2+\frac{4}{k\big|1-\big(\frac{1}{2}\big)^{\frac{1}{k}}e^{\frac{\pi}{k}\mathrm{i}}\big|}}\\
&\leqslant \sqrt{\Big(\frac{k-1}{k}\Big)^2+\frac{4}{k} \Big(\frac{5k}{16}+\frac{2}{k}\Big)}\\
&=\sqrt{\frac{9}{4}-\frac{2k-9}{k^2}}\leqslant \frac{3}{2}
\end{split}
\end{equation}
for all $k\geqslant 5$, where the last inequality follows from $k\geqslant 5$.

 Thus we derive  by (\ref{estimate1}) that
\begin{align*}
|z_\infty|&  \mathlarger{\geqslant \frac{\Big|\sqrt{(k-1)^2+4k\beta^{k+1}}\Big|}{2k|\beta^k|}-\frac{k-1}{2k|\beta^k|}}\\
& \mathlarger{=\frac{1}{2|\beta|^k} \bigg|\Big(\frac{k-1}{k}\Big)^2+\frac{4\beta^{k+1}}{k}\bigg|^{\frac{1}{2}} -\frac{k-1}{2k|\beta|^k}}\\
&\geqslant \frac{1}{2|\beta|^k}-\frac{k-1}{2k|\beta|^k}=\frac{1}{2k|\beta|^k}
\end{align*}
for all $k\geqslant 3$. On the other hand, using (\ref{beta1}) and  (\ref{estimate2}) we get
 \begin{align*}
|z_\infty|& \mathlarger{\leqslant \frac{1}{2|\beta|^k} \bigg|\Big(\frac{k-1}{k}\Big)^2+\frac{4\beta^{k+1}}{k}\bigg|^{\frac{1}{2}} +\frac{k-1}{2k|\beta|^k}}\\
&\leqslant \frac{3}{4|\beta|^k}+\frac{1}{2|\beta|^k}=\frac{5}{4|\beta|^k}
\end{align*}
for all $k\geqslant 5$.  Moreover,  by (\ref{beta2}), (\ref{k2}) and some elementary calculations we obtain that
\begin{align*}
|\beta|^{-k}&=\bigg(\Big|1-\Big(\frac{1}{2}\Big)^{\frac{1}{k}}e^{\frac{\pi}{k}\mathrm{i}}\Big|^2\bigg)^{\frac{k}{2(k+1)}}\leqslant \bigg[\frac{5}{2}\Big(\frac{5k}{16}+\frac{2}{k}\Big)^{-2}\bigg]^{\frac{k}{2(k+1)}}\leqslant \frac{39}{50}
\end{align*}
if $k\geqslant 6$, which gives us that
\begin{align*}
|z_\infty|\leqslant \frac{5}{4|\beta|^k}\leqslant \frac{5}{4}\times\frac{39}{50}=\frac{39}{40} \ \ \ \ \ \ \big(k\geqslant 6\big).
\end{align*}

For $k=3$ or $k=4$ or $k=5$, by (\ref{beta2}) we can compute directly that  $|z_\infty|< \frac{49}{50}$ in these three cases.
To summarize, we get the following estimate for $|z_\infty|$:
$$\frac{1}{2k|\beta|^k} \leqslant |z_\infty|< \frac{49}{50}$$
when $k\geqslant 3$. Observing that the above estimations  are also valid for $|w_\infty|$, so we obtain
$$\frac{1}{2k|\beta|^k} \leqslant |w_\infty| <\frac{49}{50} \ \ \ \ \ \ \big(k\geqslant 3\big). $$

For simplicity we denote the positive constant $\frac{1}{2k|\beta|^k}$ by $c_k$ for $k\geqslant 3$, then we conclude by the above arguments that
\begin{align*}
0<c_k\leqslant|z_\infty|<\frac{49}{50}<1
\end{align*}
and
\begin{align*}
0<c_k\leqslant|w_\infty|<\frac{49}{50}<1
\end{align*}
for every integer $k\geqslant 3$. Since $z_\infty\neq w_\infty$,  the polynomial $z^2p'(z)-1$ has exactly two different zeros in the unit disk $\mathbb D$ for each fixed integer $k\geqslant 3$.

Let  $$\lambda_\infty=\frac{1}{z_\infty}+p(z_\infty).$$
Next we show that  $z_\infty $ is a multiple zero of the polynomial
\begin{align*}
F(z)&:=1+z[p(z)-\lambda_\infty]\\
&=(z-\alpha)(z+\beta)^k-\lambda_\infty z
\end{align*}
in the unit disk $\mathbb D$. By the definition of $p$ and the fact that $$1=z^2_\infty p'(z_\infty),$$ we get
\begin{equation}\label{lambda infinity}
\begin{split}
\lambda_\infty &=z_\infty p'(z_\infty)+p(z_\infty).
\end{split}
\end{equation}
From the definition of $\lambda_\infty$,  we immediately obtain $$F(z_\infty)=z_{\infty}\Big(\frac{1}{z_\infty}+p(z_\infty)-\lambda_\infty\Big)=0.$$
 In order to show that $z_\infty$ is a  multiple  zero of $F$, we calculate
 \begin{align*}
 F'(z)=p(z)+zp'(z)-\lambda_\infty,
 \end{align*}
 to obtain
 \begin{align*}
 F'(z_\infty)&=p(z_\infty)+z_\infty p'(z_\infty)-\lambda_\infty\\
 &=0,
  \end{align*}
 where the last equality comes from (\ref{lambda infinity}). This gives us that
 $$F(z_\infty)=F'(z_\infty)=0,$$
 so $z_\infty\in \mathbb D$ is a multiple zero of $F$, as desired.

Using the same method as above, we can show that the polynomial $$G(z):=1+z[p(z)-\mu_\infty]$$ also has a multiple zero $w_\infty$ in the unit disk, where
 $\mu_\infty$ is defined by
\begin{align*}
\mu_\infty&:=\frac{1}{w_\infty}+p(w_{\infty})=\frac{w_\infty^2p'(w_\infty)}{w_\infty}+p(w_{\infty})\\
&=w_\infty p'(w_\infty)+p(w_\infty).
\end{align*}

 Last we will  show that  if $w$ is a root of  $z^2p'(z)=\frac{n+2}{n+1}$ in $\mathbb D$ and
  $$\lambda=\frac{1}{w}+p(w),$$ then the equation $$1+z[p(z)-\lambda]=0$$ has at least two  roots (counting multiplicities) in $\mathbb D$  for all $n$ large enough.  To do this, we observe that  the polynomials  $z^2p'(z)-\frac{n+2}{n+1}$ and $z^2p'(z)-1$ have only one  different coefficient, and  the absolute value of the difference is given by $\frac{n+2}{n+1}-1$, which tends to $0$  as $n\rightarrow \infty.$
 As shown above, $z^2p'(z)-1$ has exactly two distinct zeros $z_\infty$ and $w_\infty$ in the unit disk, Lemma \ref{continuous dependence} guarantees that there is a positive integer $N_0$ such that $z^2p'(z)-\frac{n+2}{n+1}$ has two distinct zeros $z_n$ and $w_n$ in $\mathbb D$ for each $n\geqslant N_0$, and moreover,
 \begin{align}\label{zn}
\lim\limits_{n\rightarrow \infty} z_n=z_\infty \ \ \ \ \ \mathrm{and} \ \ \ \ \  \lim\limits_{n\rightarrow \infty} w_n=w_\infty.
\end{align}

Letting $$\lambda_n=\frac{1}{z_n}+p(z_n)$$
and
$$F_n(z)=1+z[p(z)-\lambda_n]$$
for each integer $n\geqslant N_0$. Recall that $$F(z)=1+z[p(z)-\lambda_\infty]$$
and $z_\infty$ is a  multiple zero of $F$ in the unit disk $\mathbb D$. We will use the zeros of $F$  to  study the distribution of zeros of each  polynomial $F_n$ for $n$ sufficiently large.

Observe that the only difference between $F_n$ and $F$ is the coefficient of the linear term. To be more precise, the absolute value of the difference is
\begin{align}\label{lambda difference}
|\lambda_n-\lambda_\infty|=\bigg|\Big[\frac{1}{z_n}+p(z_n)\Big]-\Big[\frac{1}{z_\infty}+p(z_\infty)\Big]\bigg|.
\end{align}
Applying Lemma \ref{continuous dependence} to $F_n$ and $F$,  there exists a small number $\tau>0$ (depending only on $k$) such that if $|\lambda_n-\lambda_\infty|<\tau$, then
 each polynomial $F_n$ has two zeros $\xi_n$ and $\widetilde{\xi}_n$ in $\mathbb D$ which satisfy that
\begin{equation}
\label{xi*}
  \left\{
   \begin{array}{c}
  \mathlarger{|\xi_n-z_\infty|<\frac{1}{4^k}},\vspace{2.68mm}\\
   \mathlarger{|\widetilde{\xi}_n-z_\infty|<\frac{1}{4^k}},\\
   \end{array}
  \right.
  \end{equation}
because $z_\infty\in \mathbb D$  is a multiple zero of $F$. Notice that $\xi_n$ does not necessarily equal $\widetilde{\xi}_n$. (If $\xi_n=\widetilde{\xi}_n$, then $\widetilde{\xi}_n$ is a  multiple root of $F_n$.)

To see that $|\lambda_n-\lambda_\infty|$ can be made as  small  as we need,  recall that we have shown
$$c_k\leqslant |z_\infty|<\frac{49}{50}$$
and $z_n\rightarrow z_\infty$ as $n\rightarrow \infty$. Thus there is a positive integer $N_1\geqslant N_0$ such that for each $n\geqslant N_1$, $z_n$ belongs to the compact subset $\big\{z: \frac{c_k}{2}\leqslant |z|\leqslant \frac{49}{50}\big\}$. Noting that
the function
$$Q(z)=\frac{1}{z}+p(z)$$
is uniformly continuous on $\big\{z: \frac{c_k}{2} \leqslant |z|\leqslant \frac{49}{50}\big\}$. For the constant $\tau>0$ chosen above,  there is a  positive integer $N_2\geqslant N_1$ such that if $n\geqslant N_2$, then  (\ref{zn}) and (\ref{lambda difference}) imply that
  $$|\lambda_n-\lambda_\infty|<\tau,$$
  as
  $$\lambda_n =\frac{1}{z_{n}}+p({z_{n}})=Q(z_n)$$
  and
 $$\lambda_\infty =\frac{1}{z_{\infty}}+p({z_{\infty}})=Q(z_\infty),$$
 as required.

 Therefore, we obtain the following  estimate for $|\xi_n|$:
 \begin{align*}
|\xi_n|&\leqslant |z_\infty|+|\xi_n-z_\infty|\\
&\leqslant |z_\infty|+\frac{1}{4^k}<\frac{49}{50}+\frac{1}{4^k}\\
&\leqslant \frac{49}{50}+\frac{1}{4^3}<1
\end{align*}
for every $n\geqslant N_2$, where the second inequality follows from (\ref{xi*}).  Moreover, we also have that
$$|\widetilde{\xi}_n|\leqslant |z_\infty|+|\widetilde{\xi}_n-z_\infty|<1$$
for $n\geqslant N_2$.
Thus the polynomial  $F_n(z)=1+z[p(z)-\lambda_n]$ has at least two  zeros $\xi_n$ and $\widetilde{\xi}_n$ in the open unit disk $\mathbb D$ for each integer $n\geqslant N_2$.

In the case that $$\mu_n=\frac{1}{w_n}+p(w_n),$$
similarly we can  use the above arguments to prove that $$1+z[p(z)-\mu_n]=0$$
has at least two roots $\eta_n$ and $\widetilde{\eta}_n$ in $\mathbb D$ for all $n$ large enough ($\eta_n$ does not necessarily equal $\widetilde{\eta}_n$). Therefore,  Lemma \ref{at least two zeros} is now proved.
\end{proof}

Now we are ready to prove the second  main theorem of this section.
\begin{proof}[\bf Proof of Theorem \ref{deg=k}] Recalling  that
 $$\Big(1-\frac{1}{\beta^{k+1}}\Big)^k=\Big(1+\frac{\alpha}{\beta}\Big)^k=-\frac{1}{2}$$
 and
 $$1+zp(z)=(z-\alpha)(z+\beta)^k,$$
 we immediately obtain
 $$\alpha^2p'(\alpha)=\frac{3}{2}.$$
Theorem \ref{deg>2} gives that $0$ is an isolated eigenvalue of the Toeplitz operator $T_{\overline{z}+p}$.

To complete the proof, we need to show that $\sigma(T_{\overline{z}+p})$ has at most finitely many isolated points. To this end, we denote $$h(z)=\overline{z}+p(z).$$
From the proof of Theorem \ref{(z)+az2+bz+c}, we have that the isolated points of $\sigma(T_{h})$ are contained in the subset
 \begin{align*}
  &\sigma_p(T_{h})\bigcap \Big\{\lambda \in \mathbb C: \lambda \notin \sigma_e(T_{h})\ \mathrm{and} \  \mathrm{index}(T_{h}-\lambda I)=0\Big\}\\
  &=\sigma_p(T_{h})\bigcap \bigg\{\lambda \in \mathbb C: \lambda \notin h(\partial \mathbb D)\ \mathrm{and} \ \mathrm{wind}\bigg(\frac{1+z[p(z)-\lambda]}{z}\bigg|_{\partial \mathbb D}, ~ 0\bigg)=0\bigg\}.
  \end{align*}
 For simplicity, we let
  $$\Lambda:=\sigma_p(T_{h})\bigcap \bigg\{\lambda \in \mathbb C: \lambda \notin h(\partial \mathbb D)\ \mathrm{and} \  \mathrm{wind}\bigg(\frac{1+z[p(z)-\lambda]}{z}\bigg|_{\partial \mathbb D}, ~0\bigg)=0\bigg\}.$$
Theorem \ref{key lemma} implies that $\Lambda$ is a subset of $\bigcup\limits_{n\geqslant 0} \Omega_n$, where for each nonnegative integer $n$,
$$\Omega_n:=\bigg\{\lambda: z_\lambda \ \text{is\ the \ root \ of}\   \ \lambda=\frac{1}{z}+p(z)\ \text{in}\  \mathbb D \ \ \mathrm{and}\ \  z_\lambda^2p'(z_\lambda)=\frac{n+2}{n+1} \mathrm{\ for\  some\ } n\in \mathbb N \bigg\}$$
is a finite set.
On the other hand,  Lemma \ref{at least two zeros} gives that $1+z[p(z)-\lambda]$ has at least two zeros in the unit disk for $\lambda \in \Omega_n$ if $n>N$. Hence, we deduce that
 $$\mathrm{wind}\bigg(\frac{1+z[p(z)-\lambda]}{z}\bigg|_{\partial \mathbb D}, ~0\bigg)\geqslant 1$$
 for every $\lambda\in \Omega_n$  with  $n>N$.
   It follows that $\Lambda$ is contained in  the finite union $\bigcup\limits_{n=0}^N \Omega_n$ of  finite sets
   $\Omega_n$, and so $\Lambda$ is a finite set. This completes the proof of Theorem \ref{deg=k}.
\end{proof}

\begin{rem}
Although Theorems \ref{deg>2} and \ref{deg=k} tell us that we can construct a polynomial $p$ with $\mathrm{deg}(p)=k$ such that $\sigma(T_{\overline{z}+p})$ has isolated points for every $k\geqslant 3$, there exists a class of high-degree polynomials $p$ such that the spectra of  Toeplitz operators $T_{\overline{z}+p}$ are all connected. More specifically, let us  consider $$\phi(z)=\overline{z}+(az^k+b)$$
with $k \geqslant 3$ and $a, b$ are complex constants. Indeed, Guan and the second author showed in \cite[Proposition 4.5]{GZ} that
$$\sigma(T_\varphi)=\varphi(\partial \mathbb D)\bigcup \Big\{\lambda\in \mathbb C: \lambda\notin \varphi(\partial \mathbb D)\ \mathrm{and}\  \mathrm{wind}\big(\varphi(\partial \mathbb D), \lambda\big)\neq 0\Big\},$$
which is a connected set in the complex plane $\mathbb C$.
\end{rem}

Let $h_k(z)=\overline{z}+p_k(z)$, where $p_k$ is the polynomial with degree $k\geqslant 3$  constructed in Theorem \ref{deg=k}. In view of Theorem  \ref{deg=k}, we end this section by discussing  the topological structure of the spectra of these Toeplitz operators $T_{h_k}$.

For $\phi\in C(\overline{\mathbb D})$, we denote its boundary function by $\phi^*$, i.e., $$\varphi^*:=\phi|_{\partial \mathbb D}.$$
Let $\mathbb T_{\phi^*}$ be the Hardy-Toeplitz operator with symbol $\phi^*$, then we always have $\sigma(\mathbb T_{\phi^*}) \subset \sigma(T_\phi)$, see \cite[Theorem 2.4]{ZZ} for the details. In addition, observe that the spectral structure of Hardy-Toeplitz operators with continuous symbols implies that $\sigma(T_{\phi})\backslash \sigma(\mathbb T_{\phi^*})$ is at most countable.
However, for the harmonic polynomial $h_k$, the following corollary tells us that the distinction between the spectra of the Bergman-Toeplitz operator $T_{h_k}$ and the corresponding Hardy-Toeplitz operator $\mathbb T_{h_k^*}$ is just finitely many isolated points.

\begin{cor}\label{relation}
For each integer $k\geqslant 3$, let $h_k$ be the harmonic polynomial mentioned above. Let $T_{h_k}$ and $\mathbb T_{h_k^*}$ be the Bergman-Toeplitz operator and the Hardy-Toeplitz operator, respectively. Then the relationship  between $\sigma(T_{h_k})$ and $\sigma(\mathbb T_{h_k^*})$ is given by
$$\sigma(T_{h_k})=\sigma(\mathbb T_{h_k^*})\cup \Lambda_k,$$
where $\Lambda_k$ is a finite subset of $\sigma_p(T_{h_k})$.
\end{cor}
\begin{proof} Since $h_k^*$ is continuous on $\partial \mathbb D$,
it follows from \cite[Corollary 7.28]{Dou} or \cite[Theorem 4.6.1]{Arv} that
\begin{align*}
\sigma(\mathbb T_{h_k^*})&=h_k^*(\partial \mathbb D) \bigcup \Big\{\lambda \in \mathbb C: \lambda \notin h_k^*(\partial \mathbb D)\ \mathrm{and} \ \mathrm{wind}\big(h_k^*(\partial \mathbb D), \lambda\big)\neq 0\Big\}\\
&=h_k(\partial \mathbb D) \bigcup \Big\{\lambda\in \mathbb C: \lambda \notin  h_k(\partial \mathbb D)\ \mathrm{and}\  \mathrm{wind}\big(h_k(\partial \mathbb D), \lambda\big)\neq 0\Big\}.
\end{align*}
Thus, from the proof of Theorem \ref{(z)+az2+bz+c} we have that
\begin{align}\label{inclusion}
\sigma(\mathbb T_{h_k})\backslash \sigma(\mathbb T_{h_k^*})=\sigma_p (T_{h_k})\bigcap \Big\{\lambda \in \mathbb C: \lambda\notin \sigma_e(T_{h_k})\ \mathrm{and} \  \mathrm{index}(T_{h_k}-\lambda I)=0\Big\}.
\end{align}
Denote the intersection on the right-hand side of (\ref{inclusion}) by  $\Lambda_k$.

However, we have shown in  Theorem \ref{deg=k} that all the isolated points of $\sigma(T_{h_k})$ are contained in $\Lambda_k$  and  $\Lambda_k$  is a finite subset of $\sigma_p(T_{h_k})$.  Therefore,  we conclude that $\sigma(T_{h_k})$ is the union of $\sigma(\mathbb T_{h_k^*})$ and finitely many (isolated) eigenvalues of  $T_{h_{k}}$, to  finish the proof of Corollary \ref{relation}.
\end{proof}

\section{Weyl's theorem for a class of Toeplitz operators}

 In this section we will show that Weyl's theorem holds for a class of Toeplitz operators on the Bergman space. More precisely, letting $q$ be an arbitrary function in $H^\infty\cap C(\overline{\mathbb D})$, we will show that Weyl's theorem holds for the Bergman-Toeplitz operator $T_{\overline{z}+q}$. To do so, we begin with some standard notations related to the Weyl spectrum.

Suppose that $T$ is  a  bounded linear operator on some Hilbert space. The Weyl spectrum $\omega(T)$ of $T$ is defined by
$$\omega(T):= \bigcap_{K~\text{is~compact}}\sigma(T+K).$$
Using the characterization in \cite{Sc}, the Weyl spectrum of $T$ can be expressed as
$$\omega(T)=\sigma_e(T)\bigcup \Big\{\lambda\in \mathbb C: \lambda \notin \sigma_e(T)\ \mathrm{and}\  \mathrm{index}(T-\lambda I)\neq 0 \Big\}.$$
Following the notation in \cite{Ber, Ber2}, for simplicity  we use $\pi_{00}(T)$ to denote the set of isolated points $\lambda$ in the spectrum which are eigenvalues of finite geometric multiplicity, i.e.,
$$0<\mathrm{dim}~\mathrm{ker}(T-\lambda I)<\infty.$$
According to  \cite{Co}, we say that an operator $T$ satisfies
Weyl's theorem if
$$\omega(T)=\sigma(T)\backslash \pi_{00}(T).$$

For Bergman-Toeplitz operators with analytic and co-analytic symbols, it is clear that these operators satisfy Weyl's theorem, because their spectra are equal to the closure of the ranges of their symbols. Furthermore, Toeplitz operators with real-valued symbols and \emph{radial symbols} (i.e., $\varphi(z)=\varphi(|z|)$ for all $z\in \mathbb D$) also satisfy Weyl's theorem, since they are normal. However, unlike Toeplitz operators on the Hardy space, there exist many Bergman-Toeplitz operators for which Weyl's theorem does not hold.
\begin{example}
Let $$\varphi(z)=\chi_{\frac{1}{2}\mathbb D}(z)e^{-\mathrm{i}(\mathrm{arg}(z))},  \ \ \ z\in \mathbb D,$$
where $\frac{1}{2}\mathbb D=\big\{z\in \mathbb C: |z|<\frac{1}{2}\big\}$.  Considering the Toeplitz operator $T_{\varphi}$ on the Bergman space $L_a^2$, we have
\begin{align*}
T_{\varphi}e_n(z)=\begin{cases}
\mathlarger{0},  \ \ \  \ \ \ \  \ \ \ \ \ \ \ \ \ \ \ \ \ \ \ \ \ \ \ \ \ \ \ \ \ \ \  n=0,\vspace{3mm}\\
\mathlarger{\frac{\sqrt{n(n+1)}}{2n+1}\Big(\frac{1}{2}\Big)^{2n}e_{n-1}(z)}, \ \ \ \ n\geqslant 1,
\end{cases}
\end{align*}
where $\{e_{n}(z)\}_{n=0}^\infty=\big\{\sqrt{n+1}z^n\big\}_{n=0}^\infty$ is the orthonormal basis of $L_a^2$.
Then the Toeplitz operator $T_\varphi$ does not satisfy Weyl's theorem.
\end{example}
Indeed, noting that $T_{\phi}$ is a compact  backward weighted shift, since
$$\lim_{n\rightarrow \infty}\frac{\sqrt{n(n+1)}}{2n+1}\Big(\frac{1}{2}\Big)^{2n}=0.$$
Using $\mathrm{(a)}$ of \cite[Proposition 27.7]{Con2}, we conclude  that the spectrum  and essential spectrum  of $T_\varphi$ are both $\{0\}$, which implies that $\omega(T_\varphi)=\{0\}$. Moreover,  we observe that
$$\mathrm{ker}(T_\varphi)=\mathrm{span}\{1\}.$$
It follows that $0$ is an isolated eigenvalue with finite multiplicity and $\pi_{00}(T_\varphi)=\{0\}$. Thus we have
$$\sigma(T_\varphi)\backslash \pi_{00}(T_\varphi)=\varnothing$$
and
$$\omega(T_\varphi)\neq \sigma(T_\varphi)\backslash \pi_{00}(T_\varphi).$$
Therefore,  the Bergman-Toeplitz operator with the symbol $$\varphi(z)=\chi_{\frac{1}{2}\mathbb D}(z)e^{-\mathrm{i}(\mathrm{arg}(z))}$$  does not satisfy Weyl's theorem.

Nevertheless, in the rest of this section, we will use the characterizations for the point spectra of Toeplitz operators in Theorems \ref{key lemma} and \ref{deg=k} to  obtain  a class of Toeplitz operators with bounded harmonic symbols on the Bergman space for which Weyl's theorem holds.
\begin{thm}\label{weyl}
 Suppose that $q$ is in the disk algebra $H^\infty\cap C(\overline{\mathbb D})$ and  let $h(z)=\overline{z}+q(z)$. Then Weyl's theorem holds for the Bergman-Toeplitz operator $T_{h}$.
\end{thm}
\begin{proof}
Recall that the spectrum of the Toeplitz operator $T_h$ can be decomposed as the following disjoint union:
\begin{align*}
 \sigma(T_h)&=h(\partial \mathbb D)\bigcup \Big\{\lambda\in \mathbb C: \lambda \notin h(\partial \mathbb D)\ \mathrm{and}\  \mathrm{wind}\big(h(\partial \mathbb D), \lambda\big)\neq 0\Big\}\bigcup \Lambda,
 \end{align*}
where $$\Lambda:=\sigma_p (T_h)\bigcap \Big\{\lambda \in \mathbb C: \lambda \notin \sigma_e(T_h)\ \mathrm{and}\ \mathrm{index}(T_{h}-\lambda I)=0\Big\}.$$
By the definition of $\omega(T_h)$ and Lemma \ref{Fredholm index}, we have
$$\omega(T_h)=h(\partial \mathbb D)\bigcup \Big\{\lambda \in \mathbb C: \lambda \notin h(\partial \mathbb D)\ \mathrm{and} \  \mathrm{wind}\big(h(\partial \mathbb D), \lambda\big)\neq 0\Big\},$$
to obtain
\begin{align}\label{pi}
\sigma(T_h)=\omega(T_h)\cup \Lambda.
\end{align}
Since all the isolated points of $\sigma(T_h)$ are contained in $\Lambda$, we immediately have $\pi_{00}(T_h)\subset \Lambda$.

On the other hand, from the proofs of Theorems \ref{key lemma} and  \ref{deg>2}  we recall that $\Lambda$ is a subset of the following countable set:
$$ \bigg\{\lambda \in \mathbb C: z_\lambda \ \text{is\ the \ root \ of}\   \ \lambda=\frac{1}{z}+q(z)\ \text{in}\  \mathbb D \ \ \mathrm{and}\ \  z_\lambda^2q'(z_\lambda)=\frac{n+2}{n+1} \mathrm{\ for\  some\ } n\in \mathbb N \bigg\}.$$
It follows from Theorem \ref{spectral picture theorem} that each eigenvalue in $\Lambda$ is an isolated point of $\sigma(T_h)$. Thus we have by the definition of $\pi_{00}(T_h)$ that $\Lambda\subset \pi_{00}(T_h)$,
hence  we get
$$\Lambda=\pi_{00}(T_h).$$

To show that $T_h$ satisfies  Weyl's theorem,  we need to consider two cases. If $T_h$ has no isolated eigenvalues with finite geometric multiplicity, then we have
$\Lambda=\pi_{00}(T_h)=\varnothing$, it follows from (\ref{pi}) that
$$\omega(T_h)=\sigma(T_h).$$
For the case of $\pi_{00}(T_h)\neq \varnothing$, we actually have
$$\omega(T_h)=\sigma(T_h)\backslash \Lambda=\sigma(T_h)\backslash \pi_{00}(T_h),$$
as desired. This completes the proof of Theorem \ref{weyl}.
\end{proof}

As every  hyponormal  operator satisfies Weyl's theorem,   the following example shows that there are  a lot  of  non-hyponormal  Bergman-Toeplitz operators with harmonic symbols for which Weyl's theorem holds.

 \begin{example}
 In order to construct a non-hyponormal Toeplitz operator with harmonic polynomial symbol on the Bergman space, we first choose complex numbers $a_1, a_2, \cdots, a_n$ $(n\geqslant 2)$ such that $a_n\neq 0$ and
 \begin{align}\label{hyponormal}
 \big|a_1+2a_2+\cdots+na_n\big|<1.
 \end{align}
Letting
$$q(z)=a_1z+a_2z^2+\cdots+a_nz^n$$
and  $h$ be the harmonic polynomial
$h(z)=\overline{z}+q(z).$ Then the Toeplitz operator $T_h$ satisfies  Weyl's theorem but it is not hyponormal on $L_a^2$.
\end{example}
 In order to show that $T_h$ defined above is not hyponormal, we shall recall a necessary condition for Toeplitz operators to be hyponormal on the Bergman space.  Let $f$ and $g$ be analytic on the closed unit disk $\overline{\mathbb D}$. It was shown in \cite{AC} and \cite{Sa} that
$$|f'(z)|\geqslant |g'(z)|$$
for all $z\in \partial \mathbb D$ if the Toeplitz operator $T_{f+\overline{g}}$ is hyponormal on $L_a^2$. But  Condition $(\ref{hyponormal})$ tells us that  $|q'(1)|<1$. Thus  $T_{h}$ is not a  hyponormal operator. On the other hand,  Theorem \ref{weyl} implies that Weyl's theorem holds for the Bergman-Toeplitz operator $T_{h}$. So we obtain a class of non-hyponormal Toeplitz operators $T_{h}$ on the Bergman space for which Weyl's theorem holds.
\vspace{3mm}
\subsection*{Acknowledgment}
We would like to thank the referee for the constructive and valuable comments and suggestions that improved the
content of this paper. This work was partially supported by  NSFC (Grant Nos.: 11531003,
11701052, 11871157). The first author was  supported by  NNSF of China (12231005) and NSF of Shanghai (21ZR1404200). The second author was  supported by the Fundamental Research Funds for the Central Universities (Grant Nos.: 2020CDJQY-A039, 2020CDJ-LHSS-003).


\vspace{5mm}
\begin{thebibliography}{99}

{\footnotesize

\bibitem{AC}P. Ahern and \v{Z}. \v{C}u\v{c}kovi\'{c}, A mean value inequality with applications to Bergman
space operators, \textsl{Pacific J. Math.}, 1996, 173(2): 295-305.

\bibitem{Arv} W. Arveson, \textsl{A Short Course on Spectral Theory}, Springer, New York,  2000.

\bibitem{Ax} S. Axler, Bergman spaces and their operators, Surveys of some recent results in operator theory, vol. I, \textsl{Pitman Res. Notes Math. Ser}., vol. 171, Longman Sci. Tech., Harlow, 1988, 1-50.


\bibitem{AxZ} S. Axler and D. Zheng,  Compact operators via the Berezin transform, \textsl{Indiana Univ. Math. J.},  1998, 47(2): 387-400.

\bibitem{Bay} F. Bayart and  \'{E}. Matheron, \textsl{Dynamics of Linear Operators}, Cambridge University Press, Cambridge, 2009.

\bibitem{Ber} S. K. Berberian, An extension of Weyl's theorem to a class of not necessarily normal operators, \textsl{Michigan Math. J.}, 1969, 16(3): 273-279.

\bibitem{Ber2} S. K. Berberian and P. Halmos, The Weyl spectrum of an operator, \textsl{Indiana Univ. Math. J.}, 1970, 20(6): 529-544.


\bibitem{Co}  L. A. Coburn,  Weyl's theorem for nonnormal operators, \textsl{Michigan Math. J.},  1966, 13(3): 285-288.


\bibitem{Con}  J. B. Conway, \textsl{A Course in Functional Analysis}, second edition, Graduate Texts in Mathematics, vol. 96, Springer-Verlag, New York, 1990.

\bibitem{Con2}J. B. Conway, \textsl{A Course in Operator Theory}, American Mathematical Society, Providence, R.I., 2000.




\bibitem{Dou} R. Douglas,  \textsl{Banach Algebra Techniques in Operator Theory}, second edition, Graduate Texts in Mathematics, vol. 179, Springer, New York, 1998.


\bibitem{Dur} P. L. Duren,  \textsl{Theory of $H^p$ Spaces}, Academic Press, New York,  2000.




\bibitem{GZ} N. Guan and X. Zhao, Invertibility of Bergman-Toeplitz operators with harmonic polynomial symbols, \textsl{Sci. China Math.}, 2020, 63(5): 965-978.




\bibitem{Nik} V. P. Havin and N. K. Nikolski (Eds.), \textsl{Linear and Complex Analysis Problem Book 3 Part I}, Lecture Notes in Mathematics 1573, Springer-Verlag, 1994.


\bibitem{McS} G. McDonald and C. Sundberg, Toeplitz operators on the disc, \textsl{Indiana Univ. Math. J.},  1979,  28(4): 595-611.


\bibitem{Ob}K. K. Oberai, On the Weyl spectrum, \textsl{Illinois J. Math.}, 1974, 18(2): 208-212.


\bibitem{Ost} A. N. Ostrowski,  \textsl{Solutions of Equations in Euclidean and Banach Spaces}, third edition, Academic Press, New York, 1973.


\bibitem {Per} C. M. Pearcy,   Some Recent Developments in Operator Theory, \textsl{Regional Conference Series in Mathematics}, vol. 36, American Mathematical Society, Providence, R.I., 1978.

\bibitem {RS} Q. I. Rahman and G. Schmeisser, \textsl{Analytic Theory of Polynomials}, London Math. Soc. Monographs (N.S.) 26, Oxford University Press, New York, 2002.

\bibitem {Sa} H. Sadraoui, Hyponormality of Toeplitz operators and composition operators, PhD thesis, Purdue University, 1992.

\bibitem {Sc} M. Schechter, Invariance of the essential spectrum, \textsl{Bull. Amer. Math. Soc.}, 1965, 71(2): 365-367.


\bibitem{Str} K. Stroethoff  and  D. Zheng,  Toeplitz and Hankel operators on Bergman spaces, \textsl{Tran. Amer. Math. Soc.},
     1992, 329(2): 773-794.


\bibitem{Str1} K. Stroethoff, The Berezin transform and operators on spaces of analytic functions, Banach Center Publ., vol. 38, Polish Academy of Sciences, Warsaw (1997), 361-380.


 \bibitem{Sua} D. Su\'{a}rez, The essential norm of operators in the Toeplitz algebra on $A^p(\mathbb B_n)$,  \textsl{Indiana Univ. Math. J.}, 2007, 56(5): 2185-2232.


\bibitem{SZ} C. Sundberg and D. Zheng, The spectrum and essential spectrum of Toeplitz operators with harmonic symbols,  \textsl{Indiana Univ. Math. J.},  2010, 59(1):  385-394.



\bibitem{Widom1}H. Widom, On the spectrum of a Toeplitz operator, \textsl{Pacific J. Math.},  1964, 14(1): 365-375.


\bibitem{Widom2} H. Widom, Toeplitz operators on $H^p$, \textsl{Pacific J. Math.},  1966, 19(3): 573-582.


\bibitem{ZZ} X. Zhao and D. Zheng, The spectrum of Bergman-Toeplitz operators with some harmonic symbols, \textsl{Sci. China Math.}, 2016, 59(4): 731-740.

\bibitem{Zhu1}K. Zhu, Positive Toeplitz operators on weighted Bergman spaces of bounded symmetric domains, \textsl{J. Operator Theory}, 1988, 20(3): 329-357.

\bibitem{Zhu} K. Zhu,  \textsl{Operator Theory in Function Spaces,} Marcel Dekker, New York, 1990.}
\end{thebibliography}
\end{document}